\documentclass{amsart}

\usepackage{enumerate}
\usepackage{amsmath, amssymb, amsthm,  float, amsfonts}
\usepackage{hyperref}
 \newtheorem{theorem}{Theorem}[section]
 \newtheorem{corollary}[theorem]{Corollary}
 \newtheorem{lemma}[theorem]{Lemma}
 \newtheorem{proposition}[theorem]{Proposition}
 \newtheorem*{question}{Question}

 \newtheorem*{conjecture}{Huppert's Conjecture}
 \theoremstyle{definition}
 \newtheorem*{notation}{Notation}
\newcommand{\C}{\mathbb{C}}
\newcommand{\Z}{\mathbb{Z}}
\newcommand{\Irr}{{\mathrm {Irr}}}
\newcommand{\cd}{{\mathrm {cd}}}
\newcommand{\Aut}{{\mathrm {Aut}}}

\begin{document}

\title[Simple Classical Groups]{ Simple classical groups of Lie type are determined by  their character degrees }

\author{Hung P. Tong-Viet}
\email{tvphihung@gmail.com}
\address{School of Mathematical Sciences,
University of KwaZulu-Natal\\
Pietermaritzburg 3209, South Africa}

\begin{abstract}
Let $G$ be a finite group. Denote by $\Irr(G)$ the set of all irreducible complex characters of
$G.$ Let $\cd(G)$ be the set of all irreducible complex character degrees of $G$ forgetting
multiplicities, that is, $\cd(G)=\{\chi(1):\chi\in \Irr(G)\}$ and let $\textrm{X}_1(G)$ be the set
of all irreducible complex character degrees of $G$ counting multiplicities. Let $H$ be a finite
nonabelian simple classical group. In this paper, we will show that if $G$ is a finite group and
$\textrm{X}_1(G)=\textrm{X}_1(H)$ then $G$ is isomorphic to $H.$ In particular, this implies that
the nonabelian simple classical groups of Lie type are uniquely determined by the structure of
their complex group algebras.
\end{abstract}

\maketitle

\section{Introduction}
All groups considered are finite and all characters are complex characters. Let $G$ be a group.
Denote by $\Irr(G)$ the set of all irreducible characters of $G.$ Let $\cd(G)$ be the set of all
irreducible character degrees of $G$ forgetting multiplicities, that is, $\cd(G)=\{\chi(1):\chi\in
\Irr(G)\}.$ The \emph{degree pattern} of $G,$ denoted by $\textrm{X}_1(G),$ is the set of all
irreducible character degrees of $G$ counting multiplicities. Observe that $\textrm{X}_1(G)$ is the
first column of the ordinary character table of $G.$ We follow \cite{atlas} for notation of
nonabelian simple groups.

In \cite[Problem $2^*$]{Brauer}, R. Brauer asked whether two groups $G$ and $H$ are isomorphic given
that two group algebras $\mathbb{F}G$ and $\mathbb{F}H$ are isomorphic for all fields $\mathbb{F}.$
This is false in general. In fact, E.C. Dade \cite{Dade} constructed two non-isomorphic metabelian
groups $G$ and $H$ such that $\mathbb{F}G\cong \mathbb{F}H$ for all fields $\mathbb{F}.$ In
\cite{Hertweck}, M. Hertweck showed that this is not true even for the integral group rings. Note
that if $\Z G\cong \Z H,$ then $\mathbb{F}G\cong \mathbb{F}H$ for all fields $\mathbb{F},$ where
$\Z$ is the ring of integer numbers. For nonabelian simple groups, W. Kimmerle obtained a positive
answer in \cite{Kimmerle02}. He outlined the proof asserting  that if $G$ is a group and $H$ is a
nonabelian simple group such that $\mathbb{F}G\cong \mathbb{F}H$ for all fields $\mathbb{F}$ then
$G\cong H.$ We now consider the same problem but only assume that the complex group algebras of the
two groups are isomorphic. We note that by Molien's Theorem, knowing the structure of the complex
group algebra is equivalent to knowing the first column of the ordinary character table.

\begin{question}\label{quest1} Which groups can be uniquely determined by the structure of their complex group
algebras?
\end{question}

For example, it was shown in \cite{Hung2} that the symmetric groups are uniquely determined by the
structure of their complex group algebras. Independently, this result was also proved by Nagl in \cite{Nagl11}. It was conjectured that all nonabelian simple
groups are uniquely determined by the structure of their complex group algebras. This conjecture was verified in \cite{Hung1,Hung3} for the
alternating groups, the sporadic simple groups, the Tits group and the
simple exceptional groups of Lie type.  As pointed out by the referee, a sketch for the case of alternating and sporadic simple groups was
already given in \cite{Kimmerle02} and the case for the alternating groups has
been worked out in \cite{Nagl08}. The following stronger conjecture was proposed by Huppert in
\cite{Hupp}.

\begin{conjecture}\label{HC}  Let $H$
be any nonabelian simple group and $G$ be a group such that $\cd(G)=\cd(H).$ Then $G\cong H\times
A,$ where $A$ is abelian.
\end{conjecture}

Obviously, if Huppert's Conjecture were true then all nonabelian simple groups would be uniquely
determined by the structure of their complex group algebras. However this conjecture is still open.
In \cite{Hupp} and his preprints, Huppert himself verified the conjecture for $L_2(q),Sz(q^2),$
$18$ out of $26$ sporadic simple groups, and several small alternating and simple groups of Lie
type.
Also T. Wakefield verified the conjecture for $L_3(q),U_3(q)$ and $S_4(q)$
in \cite{Wake, Wake2}. Recently, Huppert's Conjecture for the remaining simple groups of Lie type
of rank $2$ was verified by T. Wakefield and the author, see \cite{HW1} and the references therein.

The main purpose of this paper is to complete the proof of the following result.

\begin{theorem}\label{main}  Let $H$ be a nonabelian
simple  group and $G$ be a finite group. If $\emph{\textrm{X}}_1(G)=\emph{\textrm{X}}_1(H),$ then
$G\cong H.$
\end{theorem}

As the alternating groups, the sporadic simple groups, the Tits
group and the simple exceptional groups of Lie type have been
handled in \cite{Hung1,Hung3}, we only need to consider the simple
classical groups of Lie type. Now Theorem \ref{main} and
\cite[Theorem~ $2.13$]{Ber1} yield:

\begin{corollary} Let $G$ be a group and let $H$ be a nonabelian
simple group. If $\C G\cong \C H,$ then $G\cong H.$
\end{corollary}

Thus all nonabelian simple groups are uniquely determined by the structure of their complex group
algebras.
We mention that abelian  groups are not determined by the structure of their complex group
algebras. In fact the complex group algebras of any two abelian groups of the same orders are
isomorphic. There are also examples of nonabelian $p$-groups with isomorphic complex group
algebras, for example the dihedral  group of order $8$ and the quaternion group of order $8.$
However it is conjectured that if $X_1(G)=X_1(H)$ and $H$ is solvable, then $G$ is also solvable.
We can ask the same question for the character degree sets instead of the degree patterns. Finally
the same question could be asked for the finite field of prime order. This problem for $p$-groups due to G. Higman is known as the {Modular Isomorphism Problem} and is still open.

\begin{notation} 
If $\cd(G)=\{s_0,s_1,\cdots,s_t\},$ where
$1=s_0<s_1<\cdots<s_t,$ then we define $d_i(G)=s_i$ for all $1\leq i\leq t.$  Then $d_i(G)$ is the
$i^{th}$ smallest degree of the nontrivial character degrees of $G.$ The
largest character degree of $G$ is denoted by $b(G).$ If $n$ is an integer then we denote by $\pi(n)$ the set of all
prime divisors of $n,$ and by $n_p,$ the largest $p$-part of $n,$ where $p$ is a prime. We will write $\pi(G)$ to denote the set of all distinct prime
divisors of the order of $G.$ If $N\unlhd G$ and $\theta\in \Irr(N),$ then the inertia group of
$\theta$ in $G$ is denoted by $I_G(\theta).$ The set of all irreducible constituents of $\theta^G$
is denoted by $\Irr(G|\theta).$  Other notation is standard.
\end{notation}

\section{Preliminaries}\label{sec1}
We refer to \cite[$13.8,13.9$]{car85} for the classification of unipotent characters and the notion
of symbols. Note that every simple group of Lie type $S$ in characteristic $p$ (excluding the Tits
group) possesses an irreducible character of degree $|S|_p,$ which is the size of the Sylow
$p$-subgroup of $S,$ and is called the \emph{Steinberg} character of $S$ and is denoted by $St_S.$
We first show that generically, the degree of the Steinberg character of $S$ is not the second
smallest nontrivial character degree of $S.$

\begin{lemma}\label{Lem1} If $S\neq {}^2F_4(2)'$ is a simple group of Lie type  in characteristic
$p$ defined over a field of size $q,$ then $|S|_p>d_2(S)$ whenever $S\neq L_2(2^f)$ and
$d_2(L_2(2^f))=2^f.$
\end{lemma}

\begin{proof} As $\cd(L_2(2^f))=\{1,2^f-1,2^f,2^f+1\},$ we have that $d_2(L_2(2^f))=2^f.$
If $S=L_2(q)$ with $q$ odd, then $d_1(L_2(q))=(q+\delta )/2$ where $q\equiv \delta \mbox{(mod
$4$)},$ and $d_2(S)=q-1.$ If $S=L_3(2), L_3(3)$ or $L_3(4),$ then the result follows from
\cite{atlas}. If $S=L_3(q)$ with $q\geq 5,$ then $S$ possesses characters of degrees $q^2+q$ and
$q^2+q+1$ which are less than $St_S(1)=q^3,$ and if $S=U_3(q)$ with $q\geq 3,$ then $S$ possesses
characters of degrees $q(q-1)$ and $q^2-q+1$ which are less than $q^3=St_S(1)$ (see \cite{Wake}).
For the remaining simple groups of Lie type, if $S$ is a simple exceptional group of Lie type then
the result follows from \cite{Lub}, and if $S$ is a simple classical group of Lie type of Lie rank
at least $2,$ the result follows from Table \ref{Tab2}. The two unipotent character degrees
together with the symbols in Table \ref{Tab2} were obtained by using \cite[13.8]{car85}.
\end{proof}

\begin{lemma}\emph{(Zsigmondy \cite{Zsi}).}\label{Zsigmondy} Let $q\geq 3,n\geq 3$ be integers.
Then $q^n-1$ has a prime $\ell$ such that $\ell\equiv 1$ \emph{(mod $n$)} and $\ell$ does not
divide $q^m-1$ for any $m<n.$
\end{lemma}
Such an $\ell$ is called a \emph{primitive prime divisor}. The following result is an easy
consequence of the previous lemma.

\begin{lemma}\label{sols} Let $n\geq 3$ be an integer and let $q\geq 3$ be an odd prime power.
Then $(q^n-1)/(q-1),(q^n+1)/(q+1)$ and $(q^n+1)/2$ cannot be $2$- powers.
\end{lemma}
\begin{proof}
As $q$ is odd and $n\geq 3,$ $q^n-1$ has a primitive prime divisor $\ell$ such that $\ell-1$ is
divisible by $n$ so that $\ell\geq n+1\geq 4$ and also by definition of primitive prime divisor, we
have that $\ell\nmid q-1$ so that $\ell\geq 5$ is a divisor of $(q^n-1)/(q-1),$ and hence
$(q^n-1)/(q-1)$ cannot be a power of $2.$ Observe that if $\ell$ is a primitive prime divisor of
$q^{2n}-1=(q^n-1)(q^n+1)$ then $\ell$ must divide $q^n+1.$ Moreover $\ell\nmid q^2-1$ so that
$\ell\nmid q+1.$ Hence $\ell\mid (q^n+1)/(q+1)$ and $\ell\geq 2n+1\geq 7.$ Thus $(q^n+1)/(q+1)$
cannot be a $2$-power. Finally, if $(q^n+1)/2$ is a $2$-power then $q^n+1$ is also a $2$-power.
Applying the same argument as above, we obtain a contradiction.
\end{proof}

The next result gives a lower bound for the largest character degree of the alternating groups. A
proof can be found in \cite[Lemma $2.2$]{Hung3}.
\begin{lemma}\label{Lem3} If $n\geq 10,$ then $b(A_n)\geq 2^{n-1}.$
\end{lemma}

We also need some results in character theory, in particular Clifford theory.
\begin{lemma}\emph{(\cite[Lemma $3$]{Hupp}).}\label{Clifford} Suppose $N\unlhd
G$ and $\theta\in {\rm{Irr}}(N).$ Let $I=I_G(\theta).$

$(a)$ If $\theta^I=\sum_{i=1}^k\varphi_i,$ $\varphi_i\in {\rm{Irr}}(I),$ then $\varphi_i^G\in
{\rm{Irr}}(G)$ and $|G:I|\varphi_i(1)\in\cd(G).$

$(b)$ If $\theta$ extends to $\psi\in {\rm{Irr}}( I),$ then $(\psi\tau )^G\in {\rm{Irr}}(G)$ for
all $\tau\in {\rm{Irr}}(I/N).$ In particular, $\theta(1)\tau(1)|G:I|\in {\rm{cd}}(G).$

$(c)$ If $\rho \in {\rm{Irr}}( I)$ such that $\rho_N=e\theta,$ then $\rho=\theta_0\tau_0,$ where
$\theta_0$ is a character of an irreducible projective representation of $I$ of degree $\theta(1)$
while $\tau_0$ is the character of an irreducible projective representation of $I/N$ of degree $e.$
\end{lemma}

\begin{lemma}\emph{(\cite[Lemma $2$]{Hupp}).}\label{Gallagher} Suppose $N\unlhd G$ and $\chi\in {\rm{Irr}}(G).$

$(a)$ If $\chi_N=\theta_1+\theta_2+\cdots+\theta_k$ with $\theta_i\in {\rm{Irr}}(N),$ then $k$
divides $|G/N|.$ In particular, if $\chi(1)$ is prime to $|G/N|$ then $\chi_N\in {\rm{Irr}}(N).$

$(b)$ \emph{(Gallagher's Theorem)} If $\chi_N\in {\rm{Irr}}(N),$ then $\chi\psi\in {\rm{Irr}}(G)$
for every $\psi\in {\rm{Irr}}(G/N).$

\end{lemma}

Let $\chi\in \Irr(G),$ $\chi$ is said to be of \emph{p-defect zero} for some prime $p$ if
$|G|/\chi(1)$ is coprime to $p,$ or equivalently $\chi(1)_r=|G|_r.$ We say that $\chi(1)$ is an \emph{isolated degree} of $G$ if
$\chi(1)$ is divisible by no proper nontrivial character degree of $G,$ and no proper multiple of
$\chi(1)$ is a character degree of $G.$

\begin{lemma}\emph{(\cite[Lemma~2.4]{Hung4}).}\label{isolated} If $S$ is a simple group of Lie type in characteristic $p$ with $S\neq {}^2F_4(2)',$
then the degree of the Steinberg character of $S$ with degree $|S|_p$ is an isolated degree of $S.$
\end{lemma}

If $G$ is a group, then we denote by $F(G)$ the Fitting subgroup of $G.$ By results of G. Michler
and W. Willems,  every simple group of Lie type has an irreducible character of $p$-defect zero for
any prime $p$ (see \cite{Willems}). Using this result, we obtain the following.

\begin{lemma}\label{Fitting} Let $H$ be a finite simple group of Lie type and $G$ be a finite
group. If $\cd(G)=\cd(H)$ and $|G|=|H|,$ then $F(G)$ is trivial.
\end{lemma}
\begin{proof}
Suppose that $\cd(G)=\cd(H)$ and $|G|=|H|.$ It suffices to show that $G$ has no minimal normal
abelian subgroups. By way of contradiction, assume that $A$ is a nontrivial minimal normal abelian
subgroup of $G.$ It follows that $A$ is an elementary abelian $r$-group for some prime $r.$ By
\cite[Theorem]{Willems}, $H$ possesses an irreducible character $\chi$ of $r$-defect zero, that is $\chi(1)_r=|H|_r$. As $\cd(G)=\cd(H),$ we deduce that $G$ has an irreducible character $\psi\in\Irr(G)$ such that
$\psi(1)=\chi(1).$  Furthermore, as $|G|=|H|,$ we obtain that $\psi(1)_r=\chi(1)_r=|H|_r=|G|_r.$ We now have that
$|G:A|_r=|G|_r/|A|<|G|_r=\psi(1)_r$ so that $\psi(1)$ cannot divide $|G:A|,$ which contradicts Ito's Theorem \cite[Theorem~6.15]{Isaacs}.
\end{proof}

The next two lemmas will be used to obtain the final contradiction in the proof of the main theorem. These results are taken from \cite{Bia}.

\begin{lemma}\label{charext} 
If $S$ is
a nonabelian simple group, then there exists a nontrivial irreducible character $\theta$ of $S$
that extends to ${\rm{Aut}}(S).$
\end{lemma}

\begin{proof} This is \cite[Theorems $2,3,4$]{Bia}.
\end{proof}

\begin{lemma}\emph{(\cite[Lemma $5$]{Bia}).}\label{extend} Let $N$ be
a minimal normal subgroup of $G$ so that $N\cong S^k,$ where $S$ is a nonabelian simple group. If
$\theta\in {\rm{Irr}}(S)$ extends to ${\rm{Aut}}(S),$ then $\theta^k\in {\rm{Irr}}(N)$ extends to
$G.$
\end{lemma}

Finally, for each simple classical group of Lie type $S,$ we list an upper bound for $b(S)$ in
Table \ref{Tab1}. These upper bounds were obtained in \cite[Theorem $2.1$]{Seitz}.

\section{Non-containment of character degree sets of simple groups}\label{sec2}
We assume the following set up. Let $q=p^a$ be a power of a prime $p$ and let $H$ be one of the
following simple classical groups:

\begin{equation*}\label{classical}
    L_n(q),~ U_n(q),~ O_{2n}^\epsilon(q)(n\geq
4),~ S_{2n}(q), O_{2n+1}(q)(n\geq 3).
\end{equation*}
Define
\[\mathcal{C}=\{L_4^{\epsilon}(2),L_4^{\epsilon}(3),L^\epsilon_5(2),S_6(2),S_6(3),
O_8^\pm(2),O_8^\pm(3),O_7(3)\}.\] 
The main purpose of this section is to prove the following result.

\begin{proposition}\label{Reduction} Let $G$ be a perfect group and let $M$ be a maximal normal subgroup of
$G$ so that $G/M$ is a nonabelian simple group. Assume that $\cd(G)=\cd(H)$ and $|G|=|H|.$ Then
$G/M$ is a simple group of Lie type in characteristic $p.$

\end{proposition}

\begin{proof} As $G/M$ is a nonabelian simple group, using the classification of finite simple
groups, we will eliminate other possibilities for $G/M$ and hence the result will follow. We remark
that as $\cd(G/M)\subseteq \cd(H),$ we obtain that $d_i(G/M)\geq d_i(H)$ for all $i,$ and also
$b(G/M)\leq b(H).$ Moreover as $|G|=|H|$ and $|G|=|G/M|\cdot |M|,$ we deduce that $|G/M|$ divides
$|H|$ and hence $\pi(G/M)\subseteq \pi(H).$

As $\cd(G)=\cd(H)$ and $H$ is a simple classical group of Lie type in characteristic $p,$  there
exists $\chi\in\Irr(G)$ such that $\chi(1)=|H|_p=|G|_p$ as $|G|=|H|.$ Let $\theta\in\Irr(M)$ be an
irreducible constituent of $\chi$ when restricted to $M,$ and let $I=I_G(\theta).$

{\bf Case 1.} $G/M$ is a sporadic simple group or the Tits group.

{\bf Subcase} $\theta$ is not $G$-invariant. Then $M\unlhd I\lneq G.$ As $\chi\in\Irr(G|\theta),$
we deduce from \cite[Theorem~6.11]{Isaacs} that $\chi=\phi^G$ for some $\phi\in\Irr(I|\theta).$ We have
that $\chi(1)=|G:I|\phi(1)=|H|_p.$ Hence $I/M$ is a proper subgroup of $G/M$ whose index is a prime
power. By \cite[Theorem~1]{Guralnick}, one of the following cases holds.

$(a)$ $G/M\cong M_{11},$ $I/M\cong M_{10}\cong A_6\cdot 2_3$ and $|G:I|=11;$

$(b)$ $G/M\cong M_{23},$ $I/M\cong M_{22}$ and $|G:I|=23.$

In both cases, we have that $|G:I|=p$ is prime and $(p,|I/M|)=1.$ As
$\chi(1)=p\phi(1)$ is a $p$-power, we deduce that $\phi(1)$ is also
a $p$-power so that $(\phi(1),|I/M|)=1.$ By Lemma \ref{Gallagher}(a),
$\phi$ is an extension of $\theta$ to $I.$ As $I/M$ is nonabelian,
there exists $\psi\in\Irr(I/M)$ such that $\psi(1)>1.$ By Lemma \ref{Clifford}(b), we obtain that
$(\phi\psi)^G(1)=|G:I|\phi(1)\psi(1)=\chi(1)\psi(1)\in\cd(G),$ which
contradicts Lemma \ref{isolated}.

{\bf Subcase} $\theta$ is $G$-invariant. Then $\chi_M=e\theta,$ where $e\geq 1$ and
$\chi(1)=|H|_p.$

Assume first that $\theta$ extends to $\theta_0\in\Irr(G).$ By
Gallagher's Theorem, we obtain that $\mu\theta_0,$ where  $\mu\in\Irr(G/M),$ are all the irreducible constituents of $\theta^G.$ Hence
$\chi=\tau\theta_0$ for some $\tau\in\Irr(G/M).$  Assume that
$\tau(1)=1.$ Since $G/M$ is nonabelian, it has an irreducible
character $\psi$ with $\psi(1)>1.$ Then
$\psi(1)\theta_0(1)=\psi(1)\tau(1)\theta_0(1)=\psi(1)\chi(1)\in\cd(G),$
contradicting Lemma \ref{isolated}. Hence $\tau(1)>1$ and so it is a
nontrivial $p$-power degree of $G/M.$ By \cite[Theorem~1.1]{Malle},
one of the following cases holds.

\begin{itemize}
\item[(i)] $G/M\in\{M_{11},M_{12}\}$ and $\tau(1)=11$ or $2^4;$

\item[(ii)] $G/M\in\{M_{24},Co_{2},Co_3\}$ and $\tau(1)=23;$

\item[(iii)] $G/M={}^2F_4(2)'$ and $\tau(1)=3^3$ or $2^{11}.$
\end{itemize}

Assume that $(G/M,\tau(1))\not\in\{
(M_{11},2^4),({}^2F_4(2)',2^{11})\}.$ By \cite{atlas},  there exists
$\psi\in\Irr(G/M)$ with $\psi(1)=m\tau(1)$ for some $m>1.$ We have
$\psi(1)\theta_0(1)=m\tau(1)\theta_0(1)=m\chi(1)$ is a degree of
$G,$ which contradicts Lemma \ref{isolated}.

Assume that $(G/M,\tau(1))\in\{
(M_{11},2^4),({}^2F_4(2)',2^{11})\}.$ Then $p=2$ and so $H$ is a
classical group in characteristic $2.$ By $(i)$ and $(iii)$ above
and the fact that $\cd(G/M)\subseteq \cd(H),$  we deduce that $H$
has two distinct nontrivial prime power degrees. By applying
\cite[Theorem~1.1]{Malle}, we obtain that $\tau(1)=|H|_2$ and
$H\cong L_n(2^a)$ or $U_n(2^a)$ with $n\geq 5$ an odd prime. We now
have that $n(n-1)a/2=4$ or $11,$ according to whether $G/M\cong
M_{11}$ or ${}^2F_4(2)',$ respectively. It follows that $n(n-1)a=8$
or $22.$ However both cases are impossible as $n=5$ or $n\geq 7.$

Thus  $\theta$ is not extendible to $G.$  Since $\chi(1)=e\theta(1)=|H|_p,$ we deduce that $e$ is a nontrivial $p$-power
and it is also a degree of a proper projective irreducible representation of $G/M.$ By
\cite[Theorem~1.1]{Malle}, we obtain that $(G/M,e)=(M_{12},2^5),(J_2,2^6)$ or $(Ru,2^{13}).$
Observe that the first case cannot occur as $M_{12}$ possesses a character of degree $2^4<2^5\leq
\chi(1)=|H|_2,$ which contradicts Lemma \ref{isolated}. Hence $G/M\cong J_2$ or $Ru$ and $H$ is of
characteristic $2.$

As the character tables of the sporadic simple groups $G/M$ under consideration and the simple classical groups $H$ in $\mathcal{C}$ are available in  \cite{atlas}, it is routine to check that  $\cd(G/M)\nsubseteq \cd(H)$ and so $\cd(G)\neq \cd(H)$ in these cases.
Therefore we can assume
that $H\not\in\mathcal{C}.$  We consider the following cases.

$(a)$ $H=L^\epsilon_n(q),q=2^a,n\geq 4.$ As $H\not\in\mathcal{C},$ it follows from
\cite[Table~II]{TZ} that $d_1(H)\geq (q^n-q)/(q+1).$ As $H\not\in\mathcal{C},$ we deduce that
$d_1(H)>14=d_1(J_2)$ so that $G/M$ cannot be isomorphic to $J_2.$ Now assume that $G/M\cong Ru.$
Then $d_1(Ru)=378.$ We have $d_1(H)>d_1(G/M)$ unless $H=L_n^\epsilon(2) (6\leq n\leq 10),$
$L_4^\epsilon(4)$ or $L_5^\epsilon(4).$ However we can check that $\pi(Ru)\nsubseteq \pi(H)$ in any
of these cases.

$(b)$ $H=S_{2n}(q),q=2^a,n\geq 3.$ It follows from \cite[Table~II]{TZ} that $d_1(H)\geq
(q^n-1)(q^n-q)/(2(q+1)).$ As $d_1(H)>14=d_1(J_2),$ we deduce that $G/M\cong Ru.$ We have
$d_1(H)>d_1(G/M)=378$ unless $H=S_{2n}(2)$ $(4\leq n\leq 5)$ or $S_6(4).$ But then
$\pi(Ru)\nsubseteq \pi(H)$ in any of these cases.

$(c)$ $H=O^-_{2n}(q),q=2^a,n\geq 4.$ As $H\not\in\mathcal{C},$ it follows from \cite[Table~II]{TZ}
that $d_1(H)\geq (q^n+1)(q^{n-1}-q)/(q^2-1).$ As $d_1(H)>14$ we deduce that $G/M\cong Ru.$  We have $d_1(H)>d_1(G/M)$ unless $H=O_8^-(2)$ or $O_{10}^-(2).$ However
$\pi(Ru)\nsubseteq \pi(H)$ in any of these cases.

$(d)$ $H=O^+_{2n}(2),n\geq 5.$ As $H\not\in\mathcal{C},$ it follows from \cite[Table~II]{TZ} that
$d_1(H)\geq (2^n-1)(2^{n-1}-1)/3.$ As $d_1(H)>14,$  we deduce that  $G/M\cong Ru.$  We have $d_1(H)>d_1(G/M)$ unless $H=O_{10}^+(2).$ But then
$\pi(Ru)\nsubseteq \pi(H)$ in this case.

$(e)$ $H=O^+_{2n}(q),q=2^a>2,n\geq 4.$ As $H\not\in\mathcal{C},$ it
follows from \cite[Table~II]{TZ} that $d_1(H)\geq
(q^n-1)(q^{n-1}+q)/(q^2-1)>d_1(Ru)>d_1(J_2),$  a
contradiction.

Thus $G/M$ cannot be a sporadic simple group nor the Tits group.

{\bf Case 2.} $G/M$ is an alternating group. Assume that $G/M\cong
A_m,$ where $m\geq 5.$ As $A_8\cong L_4(2),$ we consider $A_8$ as a
simple group of Lie type in characteristic $2.$

Assume first that $5\leq m\neq 8\leq 10.$ Using \cite{atlas}, we can
see that $\cd(A_m)\nsubseteq \cd(H)$ for any $H\in \mathcal{C}.$
Hence we assume that $H\not\in\mathcal{C}.$ Using \cite[Table
$\textrm{II}$]{TZ}, we observe that $d_1(H)\geq 10$ so that
$d_1(A_m)\leq 9<d_1(H)$ for $5\leq m\neq 8\leq 10,$ which is
impossible as $\cd(A_m)\subseteq \cd(H).$ 

Thus we can assume $m\geq
11.$ Hence $\{7,11\}\subseteq \pi(A_m).$ If
$H\in\mathcal{C},$ then $\{7,11\}\not\subseteq\pi(H)$ so that
$\pi(A_m)\nsubseteq \pi(H),$ which is a contradiction. We now assume
that $H\not\in\mathcal{C}.$ We outline our general argument here. As
$m\geq 11,$ we obtain that $d_1(A_m)=m-1.$ As
$\cd(G/M)=\cd(A_m)\subseteq\cd(H),$ we deduce that $d_1(A_m)=m-1\geq
d_1(H).$ By Lemma \ref{Lem3} we have that $b(A_m)\geq 2^{m-1}$ so
that $b(A_m)\geq 2^{d_1(H)}.$ Moreover as $b(A_m)\leq b(G)=b(H),$ we
deduce that $b(A_m)\leq q^{N(H)+1},$ where $|H|_p=q^{N(H)}$ by Table
\ref{Tab1}. Therefore we obtain $q^{N(H)+1}\geq 2^{d_1(H)}.$ However
for each possibility of $H,$ we can check that the latter inequality
cannot happen. For example, assume $H=L_n(q),q=p^a,n\geq 4.$ As
$H\not\in\mathcal{C},$ it follows from \cite[Table~II]{TZ} that
$d_1(H)\geq (q^n-q)/(q-1),$ and hence $q^{n(n-1)/2+1}\geq
2^{(q^n-q)/(q-1)}.$ It is routine to check that for any pairs
$(n,q)\not\in\{(4,2),(4,3),(5,2)\},$ the latter inequality cannot
happen.

{\bf Case 3.} $G/M\cong S$ is a simple group of Lie type in
characteristic $r\neq p.$ As $\textrm{cd}(G/M)\subseteq
\textrm{cd}(H),$ we deduce that $H$ possesses two distinct
irreducible characters $\chi_i,i=1,2$ with $\chi_1(1)=St_S(1)=|S|_r$
and $\chi_2(1)=St_H(1)=|H|_p.$ It follows from \cite[Theorem
$1.1$]{Malle} that one of the following cases holds.

$(a)$ $H=L_n(q),q\geq 3,n$ is odd prime, $(n,q-1)=1$ and $\chi_1(1)=(q^n-1)/(q-1).$ Then
$|S|_r=(q^n-1)/(q-1)$ is a prime power and $n$ is an odd prime and thus $n\geq 5.$ By \cite[Table
$\textrm{IV}$]{TZ}, we have $d_2(H)\geq (q^n-1)/(q-1)$ so that $d_2(H)\geq |S|_r$ and hence
 $d_2(S)\geq d_2(H)\geq |S|_r.$ By Lemma \ref{Lem1}, we have
that $S=L_2(2^f)$ for some $f\geq 2.$ Hence $2^f=(q^n-1)/(q-1),$ where $n\geq 5.$ It follows that
$q$ is odd, $n\geq 5$ and so the latter equation cannot happen by Lemma \ref{sols}.

$(b)$ $H=U_n(q),n$ is odd prime, $(n,q+1)=1$ and $\chi_1(1)=(q^n+1)/(q+1).$ As $n\geq 4$ is an odd
prime, we deduce that $n\geq 5$ is an odd prime. By \cite[Table $\textrm{V}$]{TZ}, we have
$|S|_r=\chi_1(1)\leq d_2(H)$ and thus as $d_2(H)\leq d_2(S)$ we obtain $d_2(S)\geq |S|_r$ so that
by Lemma \ref{Lem1}, $S=L_2(2^f)$ with $f\geq 2,$ and thus $2^f=(q^n+1)/(q+1).$ Hence $q$ is odd,
$n\geq 5$ and $2^f=(q^n+1)/(q+1),$ which contradicts Lemma \ref{sols}.

 $(c)$ $H=S_{2n}(q),q=p^a,p$ is odd prime, $an$ is  $2$-power and $\chi_1(1)=(q^n+1)/2.$
 By \cite[Theorem $5.2$]{TZ} we see that
$|S|_r=\chi_1(1)\leq d_2(H)$ and thus as $d_2(H)\leq d_2(S)$ we obtain that $|S|_r\leq d_2(S),$ so
that by Lemma \ref{Lem1}, $S=L_2(2^f)$ with $f\geq 2,$ and then $2^f=(q^n+1)/2,$ where $n\geq 3$
and $q$ is odd, contradicting Lemma \ref{sols}.

$(d)$ $H=S_{2n}(3),n\geq 3$ is a prime and $\chi_1(1)=(3^n-1)/2.$ By \cite[Theorem $5.2$]{TZ} we
see that $|S|_r=\chi_1(1)\leq d_1(H)$ and thus as $d_1(H)\leq d_1(S)$ we obtain $d_1(S)\geq |S|_r,$
which is impossible by Lemma \ref{Lem1}.

$(e)$ $H=S_6(2)$ and $\chi_1(1)\in\{7,3^3\}.$  Assume that $\chi_1(1)=|S|_r=7.$ Then $S$ must be a
simple group of Lie type of Lie rank $1$ in characteristic $7.$ The only possibility for $S$ is
$L_2(7).$ However $\cd(L_2(7))\nsubseteq \cd(S_6(2)).$ Finally, if  $\chi_1(1)=|S|_r=3^3,$ then $S$
is a simple group of Lie type of Lie rank at most $2$ in characteristic $3.$ Moreover  $|S|$
divides $|S_6(2)|=2^9\cdot 3^4\cdot 5\cdot 7.$ The only possibility for $S$ is $U_3(3).$ However
$\cd(U_3(3))\nsubseteq \cd(S_6(2)).$ 

Thus $G/M$ is a simple group of Lie type in characteristic
$p.$ The proof is now complete.
\end{proof}

\section{Proof of Theorem \ref{main}}\label{sec3}
In view of \cite{Hung1,Hung3}, we can assume that $H$ is a simple classical group of Lie type in
characteristic $p.$ Moreover by \cite{Hupp,Hupp06,Wake,Wake2} we can assume that $H$ is one of the
following simple groups: $L_n^\epsilon(q),O_{2n}^\epsilon(q),n\geq 4,S_{2n}(q),O_{2n+1}(q),n\geq
3,$ where $q$ is a power of $p.$ Assume that $G$ is a group with $\textrm{X}_1(G)=\textrm{X}_1(H).$ As $H$ is simple, it has
exactly one linear character which is the trivial character so that $G$ also possesses a unique
linear character and hence $G$ must be perfect. Now let $M$ be a maximal normal subgroup of $G.$ It
follows that $G/M$  is a nonabelian simple group. As $\textrm{X}_1(G)=\textrm{X}_1(H),$ we deduce
that $|G|=|H|$ and $\textrm{cd}(G)=\textrm{cd}(H).$ By Proposition \ref{Reduction}, we obtain that
$G/M$ is a simple group of Lie type in characteristic $p.$

Assume first that $M=1.$ Then $G$ is a simple group of Lie type in the same characteristic $p$ as
that of $H,$ $\cd(G)=\cd(H)$ and $|G|=|H|.$ By Artin's Theorem \cite[Theorem~5.1]{Kimmerle}, we
have that $\{G,H\}=\{S_{2n}(q),O_{2n+1}(q)\},n\geq 3,q$ odd or $\{G,H\}=\{L_{4}(2),L_{3}(4)\}.$
Using \cite{atlas}, we can check that $\cd(L_4(2))\neq \cd(L_3(4)).$ Hence it suffices to show that
$\cd(S_{2n}(q))\neq \cd(O_{2n+1}(q)),$ where $n\geq 3$ and $q$ is odd. We will show that
$\textrm{\textrm{cd}}(S_{2n}(q))\nsubseteq \textrm{\textrm{cd}}(O_{2n+1}(q)).$ Now it is well-known
that the symplectic group $Sp_{2n}(q)$ has four distinct irreducible Weil characters denoted by
$\eta_n,\eta_n^*$ of degree $(q^n-1)/2,$ and $\xi_n,\xi_n^*$ of degree $(q^n+1)/2$ (see
\cite{TZ1}). Let $\textbf{j}$ be the unique central involution in $Sp_{2n}(q).$ Then
$S_{2n}(q)\cong Sp_{2n}(q)/\langle \textbf{j}\rangle.$ Moreover let $\delta=(-1)^{(q-1)/2}.$ By
\cite[Lemma $2.6(i)$]{TZ1} we have that if $\chi\in \{\xi_n,\xi_n^*\},$ then
$\chi(\textbf{j})=\delta^n\chi(1).$ If $\chi\in \{\eta_n,\eta_n^*\},$ then
$\chi(\textbf{j})=-\delta^n\chi(1).$ Remark that if $\chi\in \textrm{Irr}(Sp_{2n}(q))$ then
$\chi\in \textrm{Irr}(S_{2n}(q))$ if and only if $\chi(\textbf{j})=\chi(1). $ It follows that
$(q^n+1)/2\in \textrm{\textrm{cd}}(S_{2n}(q))$ if  $n(q-1)/2$ is even; and  $(q^n-1)/2\in
\textrm{\textrm{cd}}(S_{2n}(q))$ if $n(q-1)/2$ is odd. Thus either $(q^n-1)/2\in
\textrm{\textrm{cd}}(S_{2n}(q))$ or $(q^n+1)/2\in \textrm{\textrm{cd}}(S_{2n}(q)).$ Using
\cite{atlas}, we can see that $\cd(S_6(3))\nsubseteq \cd(O_7(3))$ and hence we can assume that $(n,q)\neq (3,3).$ By using \cite[Theorem $6.1$]{TZ}, we have that
$d_1(O_{2n+1}(q))>(q^n+1)/2>(q^n-1)/2$ so that $\cd(S_{2n}(q))\nsubseteq \cd(O_{2n+1}(q)).$

Therefore $M\neq 1.$ Let $K\leq M$ be a minimal normal subgroup of $G.$ By Lemma \ref{Fitting}, we
deduce that $K$ is nonabelian so that $K\cong S^k,$ where $k\geq 1$ and $S$ is a nonabelian simple
group. By Lemma \ref{charext}, $S$ possesses a nontrivial irreducible character $\theta$ which is
extendible to $\Aut(S).$ By Lemma \ref{extend}, $\theta^k\in\Irr(K)$ extends to $\varphi\in\Irr(G),$ with
$\varphi(1)=\theta(1)^k>1.$ Now by Gallagher's Theorem, we
have that $\varphi\psi\in\Irr(G)$ for any $\psi\in\Irr(G/K).$ As $G/M$ is a simple group of Lie
type in characteristic $p,$ it has an irreducible character $\gamma$ of degree $|G/M|_p.$ Hence $\gamma(1)$ is a
nontrivial $p$-power degree of $G.$ By Lemma \ref{isolated}, we deduce that
$\gamma(1)=St_H(1)=|G|_p.$ As $\Irr(G/M)\subseteq\Irr(G/K),$ we deduce that
$\varphi\gamma\in\Irr(G)$ and hence $\varphi(1)\gamma(1)=\varphi(1)|H|_p\in\cd(G),$ 
contradicting Lemma \ref{isolated}. This contradiction proves the theorem.

\subsection*{Acknowledgment} The author would like to thank  the reviewer for the careful reading of the manuscript and pointing out references \cite{Nagl11, Nagl08}.  His or her comments and suggestions are very helpful. The author is grateful to Dr. Hung N. Nguyen and Prof. J. Moori for their help  in preparation of this work.

\begin{table}
 \begin{center}
  \caption{Upper bounds for the largest degree of simple classical groups}\label{Tab1}
  \begin{tabular}{c|c|c|c|c}
   \hline
 $S$ & $L_n^\epsilon(q)$&$S_{2n}(q)$&$O_{2n+1}(q)$&$O_{2n}^{\epsilon}(q)$\\\hline
 ${b}(S)\leq $& $q^{n(n-1)/2+1}$&$q^{n^2+1}$&$q^{n^2+1}$&$q^{n(n-1)+1}$\\\hline
  \end{tabular}
 \end{center}
\end{table}

\begin{table}
 \begin{center}
  \caption{Some unipotent characters of simple classical groups} \label{Tab2}
  \begin{tabular}{l|l|r}
   \hline
   $S=S(q)$  & Symbol &Degree\\ \hline
   $L_n(q),n\geq 4$ & $(1,n-1)$&$\dfrac{q^{n}-q}{q-1}$\\
   &$(2,n-2)$&$\dfrac{(q^n-1)(q^{n-1}-q^2)}{(q-1)(q^2-1)}$\\

   $U_n(q),n\geq 4$ & $(1,n-1)$&$\dfrac{q^{n}+(-1)^n q}{q+1}$\\
   &$(2,n-2)$&$\dfrac{(q^n-(-1)^n)(q^{n-1}+(-1)^n q^2)}{(q+1)(q^2-1)}$\\

   $S_{2n}(q),O_{2n+1}(q),n\geq 2$&$\binom{0\:1\:n}{\:-\:}$&$\dfrac{(q^n-1)(q^n-q)}{2(q+1)}$\\
   &$\binom{0\:1}{\:n\:}$&$\dfrac{(q^{2n}-1)(q^n+q)}{2(q^2-1)}$\\

   $O_{2n}^+(p^b),n\geq 4$&$\binom{n-1}{1}$&$\dfrac{(q^n-1)(q^{n-1}+q)}{q^2-1}$\\
   &$\binom{1\:n}{0\:1}$&$\dfrac{q^{2n}-q^2}{q^2-1}$\\

   $O_{2n}^-(p^b),n\geq 4$&$\binom{1\:n-1}{\:-\:}$&$\dfrac{(q^n+1)(q^{n-1}-q)}{q^2-1}$\\
   &$\binom{0\:1\:n}{\:1\:}$&$\dfrac{q^{2n}-q^2}{q^2-1}$\\

   \hline
\end{tabular}
\end{center}
\end{table}

\end{document}